\documentclass[draft]{amsart}

\usepackage{color}

\theoremstyle{plain}
\newtheorem{lem}{Lemma}
\newtheorem{cor}[lem]{Corollary}
\newtheorem{prop}[lem]{Proposition}
\newtheorem{thm}[lem]{Theorem}
\newtheorem{conj}[lem]{Conjecture}

\theoremstyle{definition}
\newtheorem{disc}[lem]{Remark}
\newtheorem{ex}[lem]{Example}

\newcommand{\Hom}{\operatorname{Hom}}

\title[Huneke-Wiegand Conjecture for numerical semigroup rings]{Huneke-Wiegand Conjecture for complete intersection numerical semigroup rings}

\author{P. A. Garc\'{\i}a-S\'anchez}
\address{Departamento de \'Algebra, Universidad de Granada, E-18071 Granada, Espa\~na}
\email{pedro@ugr.es}

\thanks{The first author is supported by the projects MTM2010-15595 and FQM-343,  FQM-5849, and FEDER funds.}

\author{M. J. Leamer}
\email{micahleamer@gmail.com}
\thanks{Part of this work was written while the second author was visiting the Universidad de Granada with a GENIL-YTR grant}

\begin{document}

\begin{abstract}
We give a positive answer to the Huneke-Wiegand Conjecture for monomial ideals over free numerical semigroup rings, and for two generated monomial ideals over complete intersection numerical semigroup rings. 
\end{abstract}

\maketitle
It is often the case that open problems in ring theory remain difficult when specialized to numerical semigroup rings. In such instances it may be beneficial to gain perspective on the problem by trying to tackle its number theoretic analog. Since the integral closure of a numerical semigroup ring is just the polynomial ring in one variable, this perspective seems all the more reasonable for problems where the case for integrally closed rings is much easier to solve.    

If $R$ is a one-dimensional integrally closed local domain and $M$ is a finitely generated torsion-free $R$-module, then $M$ is free if and only if $M\otimes_R\Hom(M,R)$ is torsion-free. This follows from either \cite[3.3]{Auslander} or from the structure theorem for modules over a principle ideal domains.  C. Huneke and R. Wiegand have conjectured that this property holds for all one-dimensional Gorenstein domains.
\begin{conj} \cite[473--474]{Huneke}
Let $R$ be a one-dimensional Gorenstein domain and let $M$ be a non-zero finitely generated $R$-module, which is not projective.  Then $M\otimes_R\Hom_R(M,R)$ has a non-trivial torsion submodule.
\end{conj}
The Huneke-Wiegand Conjecture is often given in another more general form as in \cite{ct}.  However, in \cite[Proposition 5.6]{ct} these two versions are shown to be equivalent.

We show that if $\Gamma$ is a free numerical semigroup in the sense of Bertin and Carbonne \cite{bc},
then monomial ideals of $k[\Gamma]$ satisfy the Huneke-Wiegand Conjecture. We also show that if $k[\Gamma]$ is a complete intersection numerical semigroup ring, then two-generated monomial ideals of $\Gamma$ satisfy the Huneke-Wiegand Conjecture. In order to prove this, we make extensive use of the concept of gluing that was introduced by Rosales (see for instance \cite[Chapter 8]{ns-book}) and inspired by Delorme in \cite{Delorme}.

In the process of proving our main results we enrich the theory of gluing numerical semigroups by showing that extensions of relative ideals behave well with respect to gluing. We also show that for every complete intersection numerical semigroup $\Gamma$ and every $s$ in $\mathbb{N}\setminus\Gamma$ there is an arithmetic sequence $(x,x+s,x+2s)$ with entries in $\Gamma$ which does not factor as the sum of two shorter arithmetic sequences with entries in $\Gamma$.  More generally we will show that this property for a symmetric numerical semigroup $\Gamma$ is inherited through gluing.

\section{fundamentals}
Let $\mathbb Z$ denote the set of integers, and $\mathbb N$ denote the set of non-negative integers. Given $X,Y\subseteq \mathbb{Z}$, write will write $X+Y$ for the set $\{x+y ~|~x \in X,\ y\in Y\}$. If $X=\{x\}$, we will also write $x+Y$ for $\{x\}+Y$. 
A \emph{numerical semigroup} $\Gamma$ is a subset of $\mathbb N$ that contains $0$, is closed under addition and satisfies $\gcd(\Gamma)=1$, where $\gcd$ stands for greatest common divisor. The condition that $\gcd(\Gamma)=1$ is equivalent to saying that the set $\mathbb{N}\setminus\Gamma$ is finite. Let $k$ be a field, and let $t$ be an indeterminate. The ring $k[\Gamma]:=\bigoplus_{n\in \Gamma} kt^n$ is called the \emph{semigroup ring} associated to $\Gamma$.


The elements in $\mathbb N\setminus \Gamma$ are called \emph{gaps} of $\Gamma$, and the cardinality of $\mathbb{N}\setminus\Gamma$ is known as the \emph{genus} of $\Gamma$, denoted by $\mathrm g(\Gamma)$. The largest integer not in $\Gamma$ is its \emph{Frobenius number}, $\mathrm F(\Gamma)$. If $x$ is an element of  $\Gamma$, then $\mathrm F(\Gamma)-x$ is never in $\Gamma$. From this fact it easily follows that $\mathrm g(\Gamma)\ge \frac{\mathrm F(\Gamma)+1}2$. If the equality holds, then we say that $\Gamma$ is \emph{symmetric}.
It can also be shown that $\Gamma$ is symmetric if and only if $\Gamma=\{x\in\mathbb{Z}|\ F(\Gamma)-x\notin\Gamma\}$; see \cite[Chapter 3]{ns-book} for this and other properties of symmetric numerical semigroups.

We say that $X\subseteq \Gamma$ generates $\Gamma$ if $\Gamma =\langle X\rangle :=\sum_{x\in X} x\mathbb N$. Every numerical semigroup has a unique finite minimal generating set,
and its cardinality is known as the \emph{embedding dimension} of $\Gamma$. For a more thorough introduction to numerical semigroups see \cite[Chapter 1]{ns-book}. 

A set $A$ of integers is said to be a \emph{relative ideal} of $\Gamma$ if $A+\Gamma\subseteq A$, and there exists an integer $x$ such that $x+A\subseteq \Gamma$. For every relative ideal $A$, there exist $x_1,\ldots, x_n$ in $A$ such that $A=\{x_1,\ldots,x_n\}+\Gamma=\bigcup_{i=1}^n (x_i +\Gamma)$. In this case, we will write $A=(x_1,\ldots, x_n)$, and we will say that $X=\{x_1,\ldots,x_n\}$ is a generating set for $A$. If no proper subset  of $X$ generates $A$, then we will refer to $X$ as the minimal generating set of $A$. The minimal generating set of a given relative ideal is necessarily unique.

If $A$ and $B$ are relative ideals of $\Gamma$ generated respectively by $X$ and $Y$, then we have the following:
\begin{itemize}
\item $A+B$ is a relative ideal of $\Gamma$ and $A+B=( x+y ~|~ x\in X,\ y\in Y)$;
\item $A\cup B$ is a relative ideal of $\Gamma$ and $A\cup B=(X\cup Y)$;
\item $A\cap B$ is also a relative ideal of $\Gamma$;
\item $A-_{\mathbb Z} B=\{ z\in \mathbb Z~|~ z+B\subseteq A\}$ is also a relative ideal of $\Gamma$.
\end{itemize}

In particular we will write 
\[ A^* = \Gamma -_\mathbb Z A = \{z\in \mathbb Z~|~ z+A\subseteq  \Gamma\}.\]

\begin{disc} \label{relations}
Let $\Gamma$ be a numerical semigroup. Let $A$, $B$ and $C$ be relative ideals of $\Gamma$, and let $x$ be an integer. Then the following relations hold:
\begin{enumerate}
\item \label{relations1} $A-_{\mathbb Z}(B\cup C)=(A-_{\mathbb Z} B)\cap (A-_{\mathbb Z} C)$; 
\item \label{relations2} $(A\cap B)-_{\mathbb Z}C=(A-_{\mathbb Z}C)\cap(B-_{\mathbb{Z}}C)$;
\item \label{relations3} $A-(x)=-x+A$;
\item \label{relations4} $(A\cup B)^*=A^*\cap B^*$; and
\item \label{relations5} $(x+A)^*=-x+A^*$, in particular, $(x)^*=(-x)$.
\end{enumerate} 

\eqref{relations1}: Suppose $x$ is in $(A-_{\mathbb Z} B)\cap (A-_{\mathbb Z} C)$. Then $x+B\subseteq A$ and $x+C\subseteq A$, so $x+(B\cup C)\subseteq A$; hence $x$ is in $A-_{\mathbb Z}(B\cup C)$. Now suppose $x$ is in $A-_{\mathbb Z}(B\cup C)$. Then $x+(B\cup C)\subseteq A$, so $x+B\subseteq A$ and $x+C\subseteq A$; hence $x$ is in $(A-_{\mathbb Z} B)\cap (A-_{\mathbb Z} C)$. 

\eqref{relations3}:  By shifting, $y$ is in $(-x+A)$ if and only of $x+y$ is in $A$.

The proof of \eqref{relations2} is similar to the proof of \eqref{relations1}. \eqref{relations4} is a special case of \eqref{relations1}. Also \eqref{relations5} is a special case of \eqref{relations3}.
\end{disc}
 
Let $\Gamma$ be a numerical semigroup. Let $A$ be a relative ideal of  $\Gamma$. 
We will say that a relative ideal $A$ is \emph{Huneke-Wiegand} if it is principal or if there exist relative ideals $P$ and $Q$ such that $P\cup Q=A$ and 
\[
((P+A^*)\cap (Q+A^*))\neq((P\cap Q)+A^*).
\]
We say that $\Gamma$ is a \emph{Huneke-Wiegand} numerical semigroup if every relative ideal of $\Gamma$ is Huneke-Wiegand.

\begin{disc} \label{rem:gen}
Let $\Gamma$ be a numerical semigroup and $A=(x_1,\hdots,x_n)$ a relative ideal of $\Gamma$. It follows from 
the equivalences in \cite[Theorem 1.4]{Leamer} that $A$ is Huneke-Wiegand if and only if there exists a partition $\{S,S'\}$ of 
$\{1,\hdots,n\}$ such that for $P=(x_i|\ i\in S)$, $Q=(x_j|\ j\in S')$ and   
$(P+A^*)\cap(Q+A^*)\neq(P\cap Q)+A^*$.
\end{disc}

\begin{disc}\label{shifting}
If $A$ is Huneke-Wiegand, then so is $x+A$, for any $x\in \mathbb Z$. 
\end{disc}

\begin{ex}
Let $\Gamma=\mathbb N$. Then every relative ideal is principal, that is, of the form $x+\mathbb N$. Thus $\mathbb{N}$ is a Huneke-Wiegand numerical semigroup. 
\end{ex}

 Given a bounded below set $S$ of integers, define $I_S$ to be the fractional ideal of $k[\Gamma]$ generated by $\{t^n~|~ n\in S\}$. 

\begin{thm}\cite[Theorem 1.4]{Leamer}\label{thm:PQ}
Let $\Gamma$ be a numerical semigroup, and let $A$ be a non-principal relative ideal of $\Gamma$. Then $I_A\otimes_{k[\Gamma]} I_{A^*}$ has a non-trivial torsion submodule if and only if $A$ is Huneke-Wiegand.
\end{thm}

\begin{cor}
Let $\Gamma$ be a numerical semigroup. Then $k[\Gamma]$ fulfills the Huneke-Wiegand Conjecture for monomial ideals if and only if $\Gamma$ is Huneke-Wiegand. 
\end{cor}

\begin{proof}
Let $A$ be a reltive ideal of $\Gamma$. We always have the isomorphism $I_{A^*}\cong \Hom_{k[\Gamma]}(I_A,k[\Gamma])$ given by sending $x$ to multiplication by $x$; see \cite[Remark 1.7]{Leamer}.  Since the map 
$[A\mapsto I_A]$ establishes a bijection between relative ideals of $\Gamma$ and fractional monomial ideals of $k[\Gamma]$, the result follows. 
\end{proof}

\section{Gluings of numerical semigroups}

Let $\Gamma$, $\Gamma_1$ and $\Gamma_2$ be numerical semigroups. We say that $\Gamma$ is the gluing of $\Gamma_1$ and $\Gamma_2$ if there exist $a_1$ in $\Gamma_2$ and $a_2$ in $\Gamma_1$ such that $\Gamma=a_1\Gamma_1+a_2\Gamma_2$. Since $\mathbb{N}\setminus\Gamma$ is finite, we must have $\gcd(a_1,a_2)=1$. 

Let $\Gamma=a_1\Gamma_1+a_2\Gamma_2$ be a gluing of two numerical semigroups $\Gamma_1$ and $\Gamma_2$. Let $A_1$ and $A_2$ be relative ideals of $\Gamma_1$ and $\Gamma_2$ respectively. Then the set $A=a_1A_1+a_2A_2$ is a relative ideal of $\Gamma$. In this case we will say that $A$ is the \emph{extension} of $A_1$ and $A_2$ to $\Gamma$. Our first goal will be to prove that if a non-principal ideal $A_1$ is Huneke-Wiegand, then so is any extension $A=a_1A_1+a_2A_2$. Also we will show that the reverse implication holds when $A_2=\Gamma_2$. Before establishing these results we will prove some elementary properties of extensions.   

Ap\'ery sets are often used in the study of numerical semigroups; see \cite{ns-book}. For a given a nonempty set of integers $S$ and $z$ a nonzero integer, the Ap\'ery set of $S$ with respect to $z$ is defined as 
\[\mathrm{Ap}(S,z)=\{ s\in S~|~ s-z\not\in S\}.\]

\begin{lem} \label{lem:mult}
Let $\Gamma_1$ and $\Gamma_2$ be numerical semigroups and $\Gamma=a_1\Gamma_1+a_2\Gamma_2$ be a gluing of 
$\Gamma_1$ and $\Gamma_2$.  Let $A=(x_1,\ldots, x_n)$ and $B=(y_1,\ldots,y_m)$ be relative ideals of $\Gamma_1$, and let $C=(z_1,\hdots ,z_h)$ and $D=(w_1,\hdots,w_{\ell})$ be a relative ideals of $\Gamma_2$. Then we have the following:
\begin{enumerate}
\item \label{lem:mult1} $a_1A+a_2C=(a_1x_i+a_2z_j|\ i=1,\hdots,n$ and $j=1,\hdots,h)$;
\item \label{lem:mult2} $a_1A+a_2C=a_1A+a_2\mathrm{Ap}(C,a_1)$;
\item \label{lem:mult3} $a_1(A\cup B)+a_2C=(a_1A+a_2C)\cup(a_1 B +a_2C)$;  
\item \label{lem:mult4} $a_1(A+B)+a_2C=(a_1A+a_2C)+(a_1B+a_2\Gamma_2)$;
\item \label{lem:mult5} $(a_1A+a_2C)\cap a_1\mathbb{Z}=a_1A+a_2\min\{za_1\in C|\ z\in\mathbb{Z}\}$;
\item \label{lem:mult6} $a_1A+a_2C=a_1B+a_2C$ if and only if $A=B$;
\item \label{lem:mult7} $a_1(A\cap B)+a_2C=(a_1A+a_2C)\cap(a_1 B +a_2C)$; 
\item \label{lem:mult8} $(a_1A+a_2C)-_{\mathbb{Z}}(a_1B+a_2D)=a_1(A-_{\mathbb Z}B)+a_2(C-_{\mathbb Z}D)$; and
\item \label{lem:mult9} $(a_1A+a_2C)^*=a_1(A^*)+a_2(C^*)$.
\end{enumerate}
\end{lem}

\begin{proof}
\eqref{lem:mult1}: This is clear.

\eqref{lem:mult2}: Clearly, $a_1A+a_2\mathrm{Ap}(C,a_1)\subseteq a_1A+a_2C$. Take $x=a_1a+ a_2c$ in $a_1A+a_2C$, with $a$ in $A$ and $c$ in $C$. There exists $d\geq 0$ such that $c-da_1$ is in $\mathrm{Ap}(C,a_1)$. Thus $x=a_1(a+da_2)+a_2(c-da_1)$ is in $a_1A+a_2\mathrm{Ap}(C,a_1)$. 

\eqref{lem:mult3}: This is standard, since unions distribute over setwise addition.

\eqref{lem:mult4}: We have the following equalities:
\[
a_1(A+B)+a_2C=a_1A+a_1B+a_2(C+\Gamma_2)=(a_1A+a_2C)+(a_1B+a_2\Gamma_2).
\]

\eqref{lem:mult5}: We have $a_1A+a_2\min\{za_1\in C|\ z\in\mathbb{Z}\}\subseteq a_1\mathbb{Z}$ and 
\[
a_1A+a_2\min\{za_1\in C|\ z\in\mathbb{Z}\}\subseteq a_1A+a_2C.
\] 
\[
\textrm{Thus}\quad a_1A+a_2\min\{za_1\in C|\ z\in\mathbb{Z}\}\subseteq (a_1A+a_2C)\cap a_1\mathbb{Z}.
\]
Conversely let $x$ be in 
$(a_1A+a_2C)\cap a_1\mathbb{Z}$. Then $x=a_1a + a_2c= a_1 z$ for some integer $z$, $a$ in $A$ and $c$ in $C$.
Since $\gcd(a_1,a_2)=1$, $a_1$ divides $c$; hence $c=da_1$ for some $d$ in $\mathbb{N}$. Let $a_1e=\min\{za_1\in C|\ z\in\mathbb{Z}\}$. Then $d\geq e$ and $x=a_1(a+a_2(d-e))+a_2(a_1e)$ is in $a_1A+a_2\min\{za_1\in C|\ z\in\mathbb{Z}\}$.

\eqref{lem:mult6}: Clearly $A=B$ implies $a_1A+a_2C=a_1B+a_2C$. Assume that  $a_1A+a_2C=a_1B+a_2C$. Using \eqref{lem:mult5} there is an integer $z$ such that 
\[
a_1A+z= (a_1A+a_2C)\cap a_1\mathbb{Z}= (a_1B+a_2C)\cap a_1\mathbb{Z}= a_1B+z.
\]
Therefore $A=B$.

\eqref{lem:mult7}: Since $(A\cap B)\subseteq A$, we have $a_1(A\cap B)+a_2C\subseteq a_1A+a_2C$.  Similarly we have 
$a_1(A\cap B)+a_2C\subseteq a_1B+a_2C$, and consequently 
\[
a_1(A\cap B)+a_2C\subseteq (a_1A+a_2C)\cap (a_1B+a_2C).
\] 

For the other inclusion, let $z$ be in $(a_1A+a_2C)\cap (a_1B+a_2C)$.  By \eqref{lem:mult2} we have $z=a_1a+a_2c=a_1b+a_2c'$ for some $a$ in $A$, $b$ in $B$ and $c, c'$ in $\mathrm{Ap}(C,a_1)$. Since $\gcd(a_1,a_2)=1$, we get that $c$ and $c'$ are congruent modulo $a_1$. However, since they are both in $\mathrm{Ap}(C,a_1)$, we have $c=c'$.  Thus $a=b$ is in $A\cap B$ and $z=a_1a+a_2c$ is in $a_1(A\cap B)+a_2C$.

\eqref{lem:mult8}: The second and eighth equalities in the next sequence are from Remark \ref{relations}~\eqref{relations1}. The fourth and fifth equalities are from \eqref{lem:mult7}. The third and seventh equalities are from Remark~\ref{relations}~\eqref{relations3}. The first and sixth equalities are straight forward, and the last equality below is from \eqref{lem:mult1}. 
\begin{align*}
\textstyle a_1(A-_{\mathbb{Z}}B)+a_2(C-_{\mathbb{Z}}D)& \textstyle=a_1(A-_{\mathbb{Z}}(\bigcup_{i=1}^m(y_i)))+a_2(C-_{\mathbb{Z}}(\bigcup_{j=1}^{\ell}(w_j)))\\
&\textstyle=a_1(\bigcap_{i=1}^m (A-_{\mathbb Z}(y_i)))+a_2(\bigcap_{j=1}^{\ell}(C-_{\mathbb{Z}}(w_j)))\\
&\textstyle =a_1(\bigcap_{i=1}^m(-y_i+A))+a_2(\bigcap_{j=1}^{\ell}(-w_j+C))\\ 
&\textstyle=\bigcap_{i=1}^m(a_1(-y_i+A)+\bigcap_{j=1}^{\ell}a_2(-w_j+C))\\
&\textstyle=\bigcap_{i,j}(a_1(-y_i+A)+a_2(-w_j+C))\\
&\textstyle=\bigcap_{i,j}((-a_1y_i-a_2w_j)+(a_1A+a_2C))\\
&\textstyle=\bigcap_{i,j}((a_1A+a_2C)-_{\mathbb{Z}}(a_1y_i+a_2w_j))\\
&\textstyle=(a_1A+a_2C)-_{\mathbb Z}\bigcup_{i,j}(a_1y_i+a_2w_j)\\
&=(a_1A+a_2C)-_{\mathbb Z}(a_1B+a_2D).
\end{align*}

\eqref{lem:mult9}: This is a special case of \eqref{lem:mult8}. 
\end{proof}

\begin{prop} \label{prop:extension}
Let $\Gamma_1$ and $\Gamma_2$ be numerical semigroups and $\Gamma=a_1\Gamma_1+a_2\Gamma_2$ be a gluing of $\Gamma_1$ and $\Gamma_2$. Let $A$ be a non-principal, Hunke-Wiegand relative ideal of $\Gamma_1$. 
Then $a_1A+a_2B$ is Huneke-Wiegand as a relative ideal of $\Gamma$ for every relative ideal $B$ of $\Gamma_2$.
\end{prop}

\begin{proof}
Let $P$ and $Q$ be relative ideals of $\Gamma_1$ such that $P\cup Q=A$ and 
\[
((P+A^*)\cap (Q+A^*))\neq((P\cap Q)+A^*).
\]
By Lemma \ref{lem:mult}~\eqref{lem:mult3}, we have $(a_1P+a_2B)\cup(a_1Q+a_2B)=a_1(P\cup Q)+a_2B=a_1A+a_2B$.  Moreover, by Lemma \ref{lem:mult}~\eqref{lem:mult7} and \eqref{lem:mult9} we have
\begin{align*}
((a_1P+a_2B)\cap(a_1Q+a_2B))+(a_1A+a_2B)^*&=a_1(P\cap Q)+a_2B+a_1A^*+a_2B^*\\
&=a_1((P\cap Q)+A^*)+a_2(B+B^*).
\end{align*}
Arguing analogously we get
\begin{align*}
((a_1P+a_2B)+(a_1A+a_2B)^*)\cap((a_1Q+a_2B)+(a_1A+a_2B)^*)\\
=(a_1(P+A^*)+a_2(B+B^*))\cap(a_1(Q+A^*)+a_2(B+B^*))\\
=a_1((P+A^*)\cap(Q+A^*))+a_2(B+B^*).
\end{align*}
Since $((P\cap Q)+A^*)\neq((P+A^*)\cap (Q+A^*))$, Lemma \ref{lem:mult}~\eqref{lem:mult6} implies that 
\[
a_1((P\cap Q)+A^*)+a_2(B+B^*)\neq a_1((P+A^*)\cap(Q+A^*))+a_2(B+B^*).
\]  
Thus $a_1A+a_2B$ is Huneke-Wiegand.
\end{proof}

At first glance it seems contradictory that Lemma \ref{lem:mult}~\eqref{lem:mult7} is essential to the proof of Proposition \ref{prop:extension}.  Lemma \ref{lem:mult}~\eqref{lem:mult7} shows that addition distributes over intersections in extensions. However, the Huneke-Wiegand property essentially says adding $A^*$ does not distribute over certain intersections.

\begin{prop}\label{mult}
Let $\Gamma_1$ and $\Gamma_2$ be numerical semigroups and $\Gamma=a_1\Gamma_1+a_2\Gamma_2$ be a gluing of $\Gamma_1$ and $\Gamma_2$. Let $A_1$ be a relative ideal of $\Gamma_1$. Then $A:=a_1A_1+a_2\Gamma_2$ is Huneke-Wiegand if and only if $A_1$ is Huneke-Wiegand.  
\end{prop}
\begin{proof}
Let $A_1=(x_1,\hdots ,x_n)$.  By Remark \ref{rem:gen}, $A_1$ is Huneke-Wiegand if and only if there exists a partition $\{S,S'\}$ of $\{1,\hdots,n\}$ such that when $P_1=(x_i|\ i\in S)$ and $Q_1=(x_j|\ j\in S')$ we have  
$(P_1+A_1^*)\cap(Q_1+A_1^*)\neq(P_1\cap Q_1)+A_1^*$. Let $P=(a_1x_i|\ i\in S)$ and $Q=(a_1x_j|\ j\in S')$.  Since a similar statement holds for $A$, it suffices to show that $(P_1+A_1^*)\cap(Q_1+A_1^*)=(P_1\cap Q_1)+A_1^*$ if and only if $(P+A^*)\cap(Q+A^*)=(P\cap Q)+A^*$.  

It follows from the relations in Lemma \ref{lem:mult} that 
\[
(P+A^*)\cap(Q+A^*)=a_1((P_1+A_1^*)\cap(Q_1+A_1^*))+a_2\Gamma_2
\]
and 
$(P\cap Q)+A^*=a_1((P_1\cap Q_1)+A_1^*)+a_2\Gamma_2$. Therefore by Lemma \ref{lem:mult} \eqref{lem:mult6} we have 
$(P_1+A_1^*)\cap(Q_1+A_1^*)=(P_1\cap Q_1)+A_1^*$ if and only if $(P+A^*)\cap(Q+A^*)=(P\cap Q)+A^*$, and  the result follows. 
\end{proof}

\begin{disc}\label{rem:mult}
Let $\Gamma=a_1\Gamma_1+a_2\Gamma_2$ be a gluing of the numerical semigroups $\Gamma_1$ and $\Gamma_2$. Let $A=(x_0,\ldots,x_n)$ be a relative ideal of $\Gamma$ such that $a_1$ divides $x_i$ for all $i$. Let $A_1=(\frac{x_1}{a_1},\ldots,\frac{x_n}{a_1})$ be a relative ideal of $\Gamma_1$. In light of Lemma \ref{lem:mult} (\ref{lem:mult1}), $A=a_1A_1+a_2\Gamma_2$. Thus Proposition \ref{mult} implies that 
 $A$ is Huneke-Wiegand if and only if $A_1$ is Huneke-Wiegand.
\end{disc}

Given a set of integers $S$ the \emph{delta set} of $S$ is $\Delta(S):=\{ s-t ~|~ s,t\in S,\ t<s\}$.

\begin{lem}\label{lem:nmid}
Let $\Gamma$ be a symmetric numerical semigroup.  Choose $g$ in $\Gamma\setminus\{0\}$ and $a$ in $\mathbb{N}$, and assume that $\Delta(\mathrm{Ap}(\Gamma,g)\setminus\{0\})\subset a\mathbb N$.  Let $A$ be a relative ideal of $\Gamma$ with minimal generating set $\{x_0,x_1,\hdots ,x_n\}$, where $x_0=0$. 
Suppose that there exists $i$ such that $x_i$ is not in $a\mathbb N$. Then $A$ is Huneke-Wiegand.
\end{lem}

\begin{proof}
Let $P=(x_j|\ x_j\in a\mathbb N)$ and let $Q=(x_h|\ x_h\not\in a\mathbb N)$. Set $F=F(\Gamma)$. We will show that $F+g$ is in $((P+A^*)\cap(Q+A^*))\setminus((P\cap Q)+A^*)$.   

Since $\{x_0,x_1,\hdots ,x_n\}$ is a minimal generating set, $x_i-x_j$ is not in $\Gamma$ for all $j\neq i$. Therefore $F-x_i+x_j$ is in $\Gamma$ for all $i\neq j$, and $F-x_i+x_j+g$ is in $\Gamma$ for all $i,j=0,1,\hdots,n$. It follows that $F-x_i+g$ is an element of $A^*$ for $i=0,\hdots,n$. Thus $F+g$ is in $(P+A^*)\cap(Q+A^*)$.

We will show by contradiction that $F+g$ is not in $(P\cap Q)+A^*$. Assume that $F+g$ is in $(P\cap Q)+A^*$. Then there exists $z$ in $A^*$ such that $F+g-z$ is in $P\cap Q$, so there exist $x_j$ in $P$ and $x_i$ in $Q$ such that $F+g-z-x_i$ and $F+g-z-x_j$ are in $\Gamma$.  Therefore $z+x_i-g$ and $z+x_j-g$ are not in $\Gamma$.  Since $z$ is in $A^*$, it follows that $z+x_i$ and $z+x_j$ are in $\Gamma\setminus\{0\}$. Hence $z+x_i$ and $z+x_j$ are in $\mathrm{Ap}(\Gamma,g)\setminus\{0\}$. By assumption $a$ divides $z+x_i-(z+x_j)=x_i-x_j$.  Since $a$ divides $x_i$, it must also divide $x_j$, which is a contradiction.  Thus 
$F+g$ is not in $(P\cap Q)+A^*$ and the result follows.
\end{proof}



\begin{thm}\label{hw-free}
Let $\Gamma=a_1\Gamma_1+a_2\mathbb N$ be a gluing of a symmetric numerical semigroup $\Gamma_1$ with $\mathbb N$.  Then $\Gamma$ is Huneke-Wiegand if and only if $\Gamma_1$ is Huneke-Wiegand.
\end{thm}
\begin{proof}
Suppose that $\Gamma_1$ is not Huneke-Wiegand.  Then there exists a relative ideal $A_1$ of $\Gamma_1$, which is not Huneke-Wiegand.  By Proposition \ref{mult}, the relative ideal $a_1A_1+a_2\mathbb{N}$ is not Huneke-Wiegand.  Thus $\Gamma$  is not Huneke-Wiegand.

Now suppose that $\Gamma_1$ is Huneke-Wiegand.
Let $A$ be a non-principal relative ideal of $\Gamma$ minimally generated by $\{x_0,\ldots, x_n\}$ for some positive integer $n$. By Remark \ref{shifting}, we may assume that $x_0=0$. 

If for every $x_i$, $a_1$ divides $x_i$, then by Remark \ref{rem:mult}, $A$ fulfills the Huneke-Wiegand property.

Assume that there exists $i$ in $\{0,\ldots,n\}$ such that $x_i$ is not in $a_1\mathbb N$. Note that $a_2$ is in $\Gamma$ and that $\mathrm{Ap}(\Gamma,a_2)\subset a_1\mathbb N$. We apply Lemma \ref{lem:nmid} with $g=a_2$ and $a=a_1$, and the result follows.
\end{proof}

Suppose that $\Gamma$ is a numerical semigroup of embedding dimension $n$ for some $n\geq 1$. We say that $\Gamma$ is \emph{free} (in the sense of Bertin and Carbonne; see \cite{bc}) if either $\Gamma=\mathbb N$ or if $\Gamma$ is the gluing of $\mathbb N$ with a free numerical semigroup of embedding dimension $n-1$.

\begin{cor}\label{free}
Free numerical semigroups are Huneke-Wiegand.
\end{cor}

Free numerical semigroups include telescopic numerical semigroups and numerical semigroups associated to an irreducible planar curve singularity.

\section{Arithmetic sequences over numerical semigroups}

Let $\Gamma$ be a numerical semigroup. An \emph{arithmetic sequence} in $\Gamma$ with step size $s$ in $\mathbb N\setminus\Gamma$ is a sequence of the form $(x,x+s,\ldots,x+ns)\subseteq \Gamma$ ($n>0$ is called the number of steps). We will denote this sequence by $(x;s;n)$. Two sequences with the same step size can be added by using set addition, and as a result we get $(x;s;n)+(y;s;m)=(x+y;s;n+m)$. The set of arithmetic sequences in $\Gamma$ with step size $s$ is therefore a semigroup, which we denote by $S_\Gamma^s$. An irreducible sequence is a sequence that cannot be expressed as the sum of two sequences. 

\begin{disc}\label{irreducible-seq-HW}
Let $\Gamma$ be a numerical semigroup, and let $A$ be a relative ideal of $\Gamma$ minimally generated by $\{0,s\}$. It follows from \cite[Proposition 4.4]{Leamer} that $A$ is Huneke-Wiegand if and only if there exists an irreducible sequence of the form $(x;s;2)$ in $S_{\Gamma}^s$.
\end{disc}

Note that a similar result to Proposition \ref{mult} can also be obtained for irreducible sequences.

\begin{lem}\label{shift-apery}
 Let $\Gamma$ be a symmetric numerical semigroup, and let $a$ be an element of $\Gamma\setminus\{0\}$. Given $s$ in $\mathbb{N}\setminus\Gamma$ and not in $\Delta(\mathrm{Ap}(\Gamma,a)))$,  the sequence 
$(\mathrm F(\Gamma)+a-s;s;2)$ is irreducible in $S_{\Gamma}^s$. Therefore for all $s$ in $\mathbb N\setminus (\bigcap_{a\in \Gamma\setminus\{0\}}\Delta(\mathrm{Ap}(\Gamma,a)))$, the relative ideal $(0,s)$ is Huneke-Wiegand.
\end{lem}
\begin{proof}
Set $F=\mathrm F(\Gamma)$. The symmetry of $\Gamma$ implies that $F-s$ is in $\Gamma$, and consequently $F-s+a$ is in $\Gamma$. Also $F+a$ and $F+a+s$ are larger than $F$, so they also belong to $\Gamma$; hence 
$(\mathrm F(\Gamma)+a-s;s;2)$ is in $S_{\Gamma}^s$.

Suppose that $(\mathrm F(\Gamma)+a-s;s;2)$ has a factor $(y;s;1)$. Then $F+a-s-y$ and  $F+a-y$ are in $\Gamma$, and the symmetry of $\Gamma$ yields that $y+s-a$ and $y-a$ are not in $\Gamma$. Since by assumption we have that $y$ and $y+s$ are in $\Gamma$, it follows that  $y$ and $y+s$ are in $\mathrm{Ap}(\Gamma,a)$; hence $s=y+s-y$ is in  $\Delta(\mathrm{Ap}(\Gamma,a))$.

If $s$ is in $\Gamma$, then $(0,s)=0+\Gamma$ is a principal ideal. If $s$ is not in $\Gamma$ and $s$ is in $\mathbb N\setminus (\bigcap_{a\in \Gamma\setminus\{0\}}\Delta(\mathrm{Ap}(\Gamma,a)))$, then there exists $a$ in $\Gamma\setminus\{0\}$ such that $s$ is not in $\Delta(\mathrm{Ap}(\Gamma,a))$. As $(\mathrm F(\Gamma)+a-s;s;2)$ is irreducible, Remark \ref{irreducible-seq-HW} implies that $(0,s)$ is Huneke-Wiegand.
\end{proof}

\begin{ex}\label{example-delta}
All computations in this example were done with the \texttt{numericalsgp} package \cite{numericalsgps}. Let $\Gamma = \langle 6, 15, 16, 25, 26\rangle$. Then $\Gamma$ is a symmetric numerical semigroup and 
$\bigcap_{a\in \Gamma\setminus\{0\}}\Delta(\mathrm{Ap}(\Gamma,a))=\{1,9,10\}$. Irreducible sequences for $s$ equal to $1$, $9$ and $10$ are $(24,25,26)$, $(6,15,24)$ and $(6,16,26)$ respectively. Hence in light of Lemma \ref{shift-apery}, every two-generated ideal of $\Gamma$ is  Huneke-Wiegand.
\end{ex}


Observe that if $\Gamma=a_1\Gamma_1+a_2\Gamma_2$ is a gluing of $\Gamma_1$ and $\Gamma_2$, then $a_1a_2$ is in $\Gamma$. By Lemma \ref{shift-apery}, if $s$ is a gap of $\Gamma$ that is not in 
$\Delta(\mathrm{Ap}(\Gamma,a_1a_2))$, then the ideal $(0,s)$ is Huneke-Wiegand. Also if $s$ is a multiple of $a_i$ for $i$ equal to 1 or 2 and $(0,\frac{s}{a_i})$ is Huneke-Wiegand in $\Gamma_i$, then Remark \ref{rem:mult} ensures that $(0,s)$ is Huneke-Wiegand in $\Gamma$. When 
$s$ is in $\Delta(\mathrm{Ap}(\Gamma,a_1a_2))$ and $s$ is not a multiple of $a_1$ or $a_2$ we will prove the existence of an irreducible sequence of the form $(x;s;2)$ in $S_{\Gamma}^s$. 
The choice of $a_1a_2$ is not arbitrary. Indeed, as we see in the next remark this Ap\'ery set has a special construction. 


\begin{disc}\label{ap-gluing}
Let $\Gamma=a_1\Gamma_1+a_2\Gamma_2$ be a gluing of $\Gamma_1$ and $\Gamma_2$. Then from 
\cite[Theorem 9.2]{ns-book} it follows that
$\mathrm{Ap}(\Gamma,a_1a_1)=a_1\mathrm{Ap}(\Gamma_1,a_2)+a_2\mathrm{Ap}(\Gamma_2,a_1)$.
\end{disc}
%
%


Let $\Gamma=a_1\Gamma_1+a_2\Gamma_2$ be a gluing of two numerical semigroups. Next we will focus on sequences with two steps having the middle term and at least one other term in 
$\mathrm{Ap}(\Gamma,a_1a_2)$. We shall see that should these sequences factor, then the original sequence can be constructed from arithmetic sequences in $\Gamma_1$ and $\Gamma_2$ that also factor.

\begin{lem}\label{Apfactor}
Let $\Gamma_1$ and $\Gamma_2$ be symmetric numerical semigroups. Let $\Gamma=a_1\Gamma_1+a_2\Gamma_2$ be a gluing of $\Gamma_1$ and $\Gamma_2$. Let $s$ be in $\mathbb{N}\setminus\Gamma$ and suppose that $(x;s;2)=(y;s;1)+(z;s;1)$ in $S_{\Gamma}^s$.
Then the following conditions hold.
\begin{enumerate}
\item \label{Ap1} If $x$ and $x+s$ are in $\mathrm{Ap}(\Gamma,a_1a_2)$, then there exist unique elements of the form $x_1,\ x_1+s_1$ in $\mathrm{Ap}(\Gamma_1,a_2)$ and $x_2,\ x_2+s_2$ in $\mathrm{Ap}(\Gamma_2,a_1)$ with $x=a_1x_1+a_2x_2$ and $s=a_1s_1+a_2s_2$.
\item \label{Ap2}  If $x+s$ and $x+2s$ are in $\mathrm{Ap}(\Gamma,a_1a_2)$, then there exist unique elements $x_1+s_1,\ x_1+2s_1$ in $\mathrm{Ap}(\Gamma_1,a_2)$ and $x_2+s_2,\ x_2+2s_2$ in $\mathrm{Ap}(\Gamma_2,a_1)$ with $x=a_1x_1+a_2x_2$ and $s=a_1s_1+a_2s_2$.
\end{enumerate}
In both of these cases there exist $y_1,\ y_2,\ z_1$ and $z_2$ such that the follwing hold:
\begin{itemize}
\item $y_1,\ y_1+s_1,\ z_1,\ z_1+s_1\in\mathrm{Ap}(\Gamma_1,a_2)$;
\item $y_2,\ y_2+s_2,\ z_2,\ z_2+s_2\in\mathrm{Ap}(\Gamma_2,a_1)$; 
\item $y=a_1y_1+a_2y_2$;
\item $z=a_1z_1+a_2z_2$; 
\item $x_1=y_1+z_1$; and 
\item $x_2=y_2+z_2$. 
\end{itemize}
Thus 
\begin{align*}a_1((y_1;s_1;1)+(z_1;s_1;1))+a_2((y_2;s_2;1)+(z_2;s_2;1))&=a_1(x_1;s_1;2)+a_2(x_2;s_2;2)\\
&=(x;s;2).
\end{align*} 
\end{lem}

\begin{proof} 
\eqref{Ap1}: Since $\Gamma=a_1\Gamma_1+a_2\Gamma_2$ is a gluing, there exist $y_1,\ z_1$ in $\Gamma_1$ and $y_2,\ z_2$ in $\Gamma_2$ such that $y=a_1y_1+a_2y_2$ and $z=a_1z_1+a_2z_2$.  Since $a_1$ and $a_2$ are relatively prime, it follows that there exists and integer $c$ such that $x_1=y_1+z_1+a_2c$ and $x_2=y_2+z_2-a_2c$.  Since $x_1$ is in $\mathrm{Ap}(\Gamma_1,a_2)$ and $y_1+z_1\in\Gamma_1$, it follows that $c\leq 0$. Similarly since $x_2$ is in $\mathrm{Ap}(\Gamma_2,a_1)$ and $y_2+z_2$ is in $\Gamma_2$, it follows that $c\geq 0$. Thus 
$x_1=y_1+z_1$ and $x_2=y_2+z_2$. As $y+(z+s)=(y+s)+z=x$ is in $\mathrm{Ap}(\Gamma,a_1a_2)$, it follows that $y,\ z,\ y+s$ and $z+s$ are in $\mathrm{Ap}(\Gamma,a_1a_2)$.  Thus $x_1,\ y_1$ and $z_1$ are in $\mathrm{Ap}(\Gamma_1,a_2)$ and $x_2,\ y_2$ and $z_2$ are in $\mathrm{Ap}(\Gamma_2,a_1)$. Since $x+s$ is in $\mathrm{Ap}(\Gamma,a_1a_2)$, there exist $s_1$ and $s_2$ in $\mathbb{Z}$ such that $x+s=a_1(x_1+s_1)+a_2(x_2+s_2)$ with 
$x_1+s_1$ in $\mathrm{Ap}(\Gamma_1,a_2)$ and $x_2+s_2$ in $\mathrm{Ap}(\Gamma_2,a_1)$.  It follows that $s=a_1s_1+a_2s_2$. 
Since $a_1$ and $a_2$ are relatively prime, it follows that there exists an integer $u$ such that $y+s$ factors as a sum $a_1(y_1+s_1+ua_2)+a_2(y_2+s_2-ua_1)$ with $y_1+s_1+ua_2$ in $\mathrm{Ap}(\Gamma_1,a_2)$ and 
$y_2+s_2-ua_1$ in $\mathrm{Ap}(\Gamma_2,a_1)$. Similarly there exists and integer $v$ such that $z+s=a_1(z_1+s_1+va_2)+a_2(z_2+s_2-va_1)$ with $z_1+s_1+va_2$ in $\mathrm{Ap}(\Gamma_1,a_2)$ and 
$z_2+s_2-va_1$ in $\mathrm{Ap}(\Gamma_2,a_1)$. By using that $x_1+s_1$ is in $\mathrm{Ap}(\Gamma_1,a_2)$ and $(y_1+s_1+ua_2)+z_1=x_1+s_1+ua_2$ in $\Gamma_1$, it follows that $u\geq 0$. Since $x_2+s_2$ is in 
$\mathrm{Ap}(\Gamma_2,a_1)$
 and $(y_2+s_2-ua_1)+z_2=x_2+s_2-ua_1$ is in $\Gamma_2$, it follows that $u\leq 0$.  Thus $u=0$. Similarly $v=0$ and Case \eqref{Ap1} follows. 

The proof of Case \eqref{Ap1} does not require that $s>0$.  If we set $x'=x+2s$ and $s'=-s$, then Case \eqref{Ap2} follows by a similar argument. 
\end{proof}

\begin{prop} \label{PropGlue}
Let $\Gamma_1$ and $\Gamma_2$ be symmetric numerical semigroups and let $\Gamma=a_1\Gamma_1+a_2\Gamma_2$ be a gluing of $\Gamma_1$ and $\Gamma_2$. Let $s$ be in $\mathbb{N}\setminus\Gamma$ such that $s$ is not in $a_1\mathbb{N}\cup a_2\mathbb{N}$.  Then $(0,s)$ is a Huneke-Wiegand relative ideal of $\Gamma$.
\end{prop}

\begin{proof}
By Lemma \ref{shift-apery} we may assume that $s$ is in $\Delta(\mathrm{Ap}(\Gamma,a_1a_2))$. Choose $u$ and $v$ in $\mathrm{Ap}(\Gamma,a_1a_2)$ such that $u-v=s$.  We may write $u$ and $v$ uniquely as 
$u=a_1u_1+a_2u_2$ and $v=a_1v_1+a_2v_2$ with $u_1,\ v_1$ in $\mathrm{Ap}(\Gamma_1,a_2)$ and $u_2,\ v_2$ in $\mathrm{Ap}(\Gamma_2,a_1)$. Let $s_1=u_1-v_1$ and $s_2=u_2-v_2$. Choose $w_1$ and $w_2$ minimal  such that 
$w_1$ and $w_1+s_1$ are in $\mathrm{Ap}(\Gamma_1,a_2)$ and $w_2$ and $w_2+s_2$ are in $\mathrm{Ap}(\Gamma_2,a_1)$. 
Also let $F_1:=F(\Gamma_1)$ and $F_2:=F(\Gamma_2)$.

Suppose that $w_1+2s_1$ is in $\Gamma_1$ and $w_2+2s_2$ is in $\Gamma_2$.  Assume that the sequence 
$(a_1w_1+a_2w_2;s,2)$ factors as $(y,y+s)+(z,z+s)$.  Then we may apply Lemma \ref{Apfactor}~\eqref{Ap1} with 
$x_1=w_1$ and $x_2=w_2$.  It follows that $y=a_1y_1+a_2y_2$ with $y_1$ and $y_1+s_1$ in 
$\mathrm{Ap}(\Gamma_1,a_2)$. Since $w_1=y_1+z_1$ and $z_1>0$, we have $y_1<w_1$.  However, this contradicts the minimality of $w_1$.  Thus $(a_1w_1+a_2w_2;s;2)$ is irreducible.  

Suppose that $w_1-s_1$ is in $\Gamma_1$ and $w_2-s_2$ is in $\Gamma_2$. Then by applying Lemma \ref{Apfactor}~\eqref{Ap2} with $x_1=w_1-s_1$ and $x_2=w_2-s_2$ a similar argument to the one above shows that $(a_1w_1+a_2w_2-s;s,2)$ is irreducible.  

Suppose that $w_1-s_1$ and $w_1+2s_1$ are in $\Gamma_1$.  By excluding previous cases we may assume that 
$w_2-s_2$ is not in $\Gamma_2$.  Thus $w_2-s_2-a_1$ is not in $\Gamma_2$. By the symmetry of $\Gamma_2$, it follows that 
$F_2-w_2-s_2+a_1$ and $F_2-w_2+a_1$ are in $\mathrm{Ap}(\Gamma_2,a_1)$ and $F_2-w_2+s_2+a_1$ is in $\Gamma_2$.  Assume that the sequence $(a_1w_1+a_2(F_2-w_2-s_2+a_1);s;2)$ factors as $(y,y+s)+(z,z+s)$. By applying Lemma \ref{Apfactor}~\eqref{Ap1} with 
$x_1=w_1$ and $x_2=F_2-w_2-s_2+a_1$ we again get elements $y_1$ and $y_1+s_1$ in $\mathrm{Ap}(\Gamma_1,a_)$ with $y_1<w_1$ contradicting the minimality of $w_1$ and implying that $(a_1w_1+a_2(F_2-w_2-s_2+a_1);s;2)$ is irreducible.

Similarly if we suppose that $w_2-s_2$ and $w_2+2s_2$ are in $\Gamma_2$ and that we are not in a previous case, then by a similar argument to the one above, it follows that $(a_1(F_1-w_1-s_1+a_2)+a_2w_2;s;2)$ is irreducible. 

We are left with two cases:
\begin{itemize} 
\item $w_1-s_1\notin\Gamma_1$ and $w_2+2s_2\notin\Gamma_2$; or 
\item $w_2-s_2\notin\Gamma_2$ and $w_1+2s_1\notin\Gamma_1$.
\end{itemize} 
Since these cases are identical up to permuting the subscripts, we may assume that $w_1-s_1$ is not in $\Gamma_1$ and $w_2+2s_2$ is not in $\Gamma_2$.  Let $h$ be the smallest integer such that either 
$w_1-s_1+ha_2$ is in $\Gamma_1$ or $w_2+2s_2+ha_1$ is in $\Gamma_2$. Suppose that $w_1-s_1+ha_2$ is in $\Gamma_1$ and $w_2+2s_2+(h-1)a_1$ is not in $\Gamma_2$.  The case where $w_1-s_1+(h-1)a_2$ is not in $\Gamma_1$ and $w_2+2s_2+(h-1)a_1$ is in $\Gamma_2$ requires a similar argument. Since $\Gamma_2$ is symmetric, $F_2-w_2-2s_2-(h-1)a_1$ is in $\Gamma_2$ and $F_2-w_2-s_2+a_1,\ F_2-w_2+a_1$ are in $\mathrm{Ap}(\Gamma_2,a_1)$.  Thus the elements
\begin{align*}
a_1(w_1-s_1+ha_2)+a_2(F_2-w_2-2s_2-(h-1)a_1),\\
a_1w_1+a_2(F_2-w_2-s_2+a_1)\\
\textrm{ and }\quad a_1(w_1+s_1)+a_2(F_2-w_2+a_1)
\end{align*}
in $\Gamma$ form a sequence  in $S_{\Gamma}^s$, which we will denote by $\alpha$. Assume that $\alpha$ factors as $(y,y+s)+(z,z+s)$. We may apply Lemma \ref{Apfactor}~\eqref{Ap2} with $x_1=w_1-s_1$ and $x_2=F_2-w_2-2s_2+a_1$. Thus we get elements $y_1,\ z_1,\ y_1+s_1$ and $z_1+s_1$ in $\mathrm{Ap}(\Gamma_1,a_2)$ all non-zero such that $w_1=y_1+(z_1+s_1)$.  It follows that $y_1<w_1$.  Since this contradicts the minimality of $w_1$, we deduce that $\alpha$ must have been irreducible and the result follows from Remark \ref{irreducible-seq-HW}.  
\end{proof}

\begin{thm}\label{hw-gluing}
Let $\Gamma=a_1\Gamma_1+a_2\Gamma_2$ be a gluing of symmetric numerical semigroups $\Gamma_1$ and $\Gamma_2$. Assume that all two-generated relative ideals of $\Gamma_1$ and of $\Gamma_2$ are Huneke-Wiegand. Then every two-generated relative ideal of $\Gamma$ is Huneke-Wiegand.
\end{thm}
\begin{proof}
Let $A$ be a two generated ideal of $\Gamma$. By Remark \ref{shifting}, we may assume that $A=(0,s)$, with $s$ in $\mathbb N\setminus \Gamma$. If  $s$ is in $a_1\mathbb N\cup a_2\mathbb N$, the result follows by Remark \ref{rem:mult}. For $s$ not in $a_1\mathbb N\cup a_2\mathbb N$, apply Proposition \ref{PropGlue}.
\end{proof}


A numerical semigroup $\Gamma$ is \emph{complete intersection} if $k[\Gamma]$ is complete intersection. C. Delorme proved in \cite{Delorme} that a numerical semigroup other than $\mathbb{N}$ is complete intersection if and only if it is the gluing of two complete intersection numerical semigroups with fewer generators. Hence by iteratively applying Theorem \ref{hw-gluing}, we get the following Corollary. Note that he also proved that a gluing is symmetric if and only if each factor is symmetric.

\begin{cor}\label{complete-intersection}
Two-generated ideals in complete intersection numerical semigroups are Huneke-Wiegand.
\end{cor}

\begin{ex}\label{ex-non-gluing}
 By using the \texttt{numericalsgps GAP} package (\cite{numericalsgps}) we checked that two-generated ideals of symmetric numerical semigroups with Frobenius number less than 69 are Huneke-Wiegand. Let $X_1$ be the set of symmetric numerical semigroups with Frobenius number less 69.  For each $i$ in $\mathbb{N}$ let $X_{i+1}$ be $X_i$ union the set of gluings of pairs of numerical semigroups  in $X_i$.  Then two generated ideals of numerical semigroups in $\cup_{i=1}^{\infty}X_i$ are Huneke-Wiegand.
\end{ex}

\section{Future work}

It seems reasonable that methods similar to those in this paper might be effective for proving slightly more general cases of the  Huneke-Wiegand Conjecture. Some natural next steps would be to answer the following questions:
\begin{itemize}
\item Do two generated ideals over complete intersection discrete valuation rings satisfy the Huneke-Wiegand Conjecture? 
\item   In \cite[(3.2)]{Herzinger}, Herzinger shows that the Huneke-Wiegand Conjecture holds for any two generated ideal $I$ over a one-dimensional local Gorenstein domain $R$, such that $\Hom_R(I,R)$ is also two generated. For numerical semigroup rings, can we remove the condition that $\Hom_R(I,R)$ is also two generated and still show that the Huneke-Wiegand Conjecture holds?
\item Is the Huneke-Wiegand property for relative ideals with more than two generators invariant under gluings of numerical semigroups?
\item If $\Gamma$ is Huneke-Wiegand do graded modules over $k[\Gamma]$ satisfy the Huneke-Wiegand Conjecture?
\end{itemize}
Another way of viewing gluings from the ring perspective is to take two numerical semigroup rings $R_1=k[\Gamma_1]$ and $R_2:=k[\Gamma_2]$.  Then 
$k[a_1\Gamma_1+a_2\Gamma_2]\cong \frac{R_1\otimes_k R_2}{(t^{a_2}\otimes 1)-(1\otimes t^{a_1})}$, where multiplication in $R_1\otimes_kR_2$ is defined component-wise.
A natural generalization from this perspective is to ask the following question.  Let  $R_1$ and $R_2$ be two commutative local domains each containing the same residue field $k$, which satisfy the Huneke-Wiegand Conjecture. If $(f\otimes 1)-(1\otimes g)$ is a regular, irreducible element of $R_1\otimes_k R_2$ what can we say about the ring 
$R:=\frac{R_1\otimes_k R_2}{(f\otimes 1)-(1\otimes g)}$? A good starting place might be to try and generalize Proposition \ref{prop:extension} to this setting: Given an $R_1$-module $M_1$ that satisfies the Huneke-Wiegand Conjecture can we show that the $R$-module $\frac{M_1\otimes_kM_2}{((t^{a_2}\otimes 1)-(1\otimes t^{a_1}))(M_1\otimes_kM_2)}$ satisfies the Huneke-Wiegand Conjecture for every $R_2$-module $M_2$.

\end{document}